\documentclass[11pt,reqno]{amsart}

\usepackage{amsthm,amsmath,amssymb}
\usepackage{graphicx}
\usepackage{float}
\usepackage[colorlinks=true,
linkcolor=blue,
urlcolor=blue,
citecolor=blue]{hyperref}
\usepackage{mathtools}
\usepackage{relsize}
\usepackage{ytableau}
\usepackage{tikz}
 \usepackage{enumerate}
\usepackage[toc,page]{appendix}
\usepackage{lscape}
\usepackage{multirow}
\usepackage{adjustbox,expl3,etoolbox} 
\usepackage[margin = 2.8cm]{geometry} 

\newtheorem{thm}{Theorem}[section]
\newtheorem{lem}[thm]{Lemma}
\theoremstyle{definition}
\newtheorem{defn}[thm]{Definition}
\newtheorem{prop}[thm]{Proposition}

\newtheorem{cor}[thm]{Corollary}
\newtheorem{exa}[thm]{Example}

\newtheorem{qst}[thm]{Question}
\newtheorem{obs}[thm]{Observation}

\DeclareMathOperator\sym{Sym}

\DeclareMathOperator\gl{GL}

\DeclareMathOperator{\pg}{PG}
\DeclareMathOperator{\ag}{AG}

\newcommand\GL[1]{\operatorname{GL}_2(\mathbb{F}_{#1})}
\newcommand\SL[1]{\operatorname{SL}_2(\mathbb{F}_{#1})}

\letcs\replicate{prg_replicate:nn}

\begin{document}	

\title[]{No Hilton-Milner type results for linear groups of degree two}

\author[Roghayeh Maleki]{Roghayeh Maleki}
\address{Department of Mathematics and Statistics, University of Regina,
	Regina, Saskatchewan S4S 0A2, Canada}\email{rmaleki@uregina.ca}

\author[A. Sarobidy Razafimahatratra]{Andriaherimanana Sarobidy Razafimahatratra}
\address{Department of Mathematics and Statistics, University of Regina,
  Regina, Saskatchewan S4S 0A2, Canada}\email{sarobidy@phystech.edu}

\begin{abstract} 
	
	A set of permutations $\mathcal{F}$ of a finite transitive permutation group $G\leq \sym(\Omega)$ is \emph{intersecting} if any pair of elements of $\mathcal{F}$ agree on an element of $\Omega$. We say that $G$ has the \emph{EKR property} if an intersecting set of $G$ has size at most the order of a point stabilizer.  Moreover, $G$ has the \emph{strict-EKR} property whenever $G$ has the EKR property and any intersecting set of maximum size is a coset of a point stabilizer of $G$. 
	
	It is known that the permutation group $\GL{q}$ acting on $\Omega_q := \mathbb{F}_q^2\setminus\{0\}$ has the EKR property, but does not have the strict-EKR property since the stabilizer of a hyperplane is a maximum intersecting set. In this paper, it is proved that the Hilton-Milner type result does not hold for $\GL{q}$ acting on $\Omega_q$. Precisely, it is shown that a maximal intersecting set of $\GL{q}$ is of maximum size. As a result, we prove the Complete Erd\H{o}s-Ko-Rado theorem for $\GL{q}$.

\end{abstract}

\subjclass[2010]{Primary 05C35; Secondary 05C69, 20B05}

\keywords{Linear groups, Erd\H{o}s-Ko-Rado type theorems, Hilton-Milner type theorems}

\date{\today}

\maketitle

\section{Introduction}

	The Erd\H{o}s-Ko-Rado (EKR) theorem \cite{erdos1961intersection} is a classical result in extremal combinatorics which asserts that if $n\geq 2k$ and $\mathcal{F}$ is a collection of $k$-subsets of $[n]:= \{1,2,\ldots,n\}$ such that $A\cap B \neq \varnothing$, for all $A,B \in \mathcal{F}$, then $|\mathcal{F}|\leq \binom{n-1}{k-1}$. Moreover, if $n\geq 2k+1$ and $|\mathcal{F}| = \binom{n-1}{k-1}$, then there exists $a\in [n]$ such that $\mathcal{F} = \left\{ A \subset [n] \mid |A|=k, a\in A \right\}$. We call these maximum intersecting sets \emph{canonical}.
	In 1967, Hilton and Milner \cite{hilton1967some} characterized the largest non-canonical intersecting sets of $k$-subsets of $[n]$. These intersecting sets are maximal (i.e., adding another element to it results in a non-intersecting set), without being maximum. Such intersecting sets are of size at most
	\begin{align}
		1 + \binom{n-1}{k-1} - \binom{n-k-1}{k-1}.\label{eq:HM}
	\end{align}
	
	Let $\sym(n)$ be the symmetric group on $[n]$. We say that two permutations $\sigma$ and $\tau$ are \emph{intersecting} if there exists $i\in [n]$ such that $\sigma(i) = \tau(i)$. Deza and Frankl proved an EKR-type result for $\sym(n)$ in \cite{Frankl1977maximum}. They proved that if $\mathcal{F} \subset \sym(n)$ is such that any two permutations of $\mathcal{F}$ are intersecting, then $|\mathcal{F}|\leq (n-1)!$. Deza and Frankl also conjectured that the sets of permutations attaining this upper bound are cosets of a point stabilizer of the permutation group $\sym(n)$. This conjecture of Deza and Frank was proved by Cameron and Ku \cite{cameron2003intersecting}, independently, by Larose and Malvenuto \cite{larose2004stable}. Another proof of this conjecture also appeared later in \cite{godsil2009new} using the Hoffman bound. The \emph{canonical intersecting} sets for the permutations of $\sym(n)$ are the cosets of a point stabilizer. That is, sets of the form $V_{i,j} = \left\{ \sigma \in \sym(n) \mid \sigma(i) = j \right\}$.
	
	In \cite{cameron2003intersecting}, Cameron and Ku conjectured a Hilton-Milner type result for the symmetric group. This conjecture came to be known as the \emph{Cameron-Ku conjecture} and it was proved by Ellis in \cite{ellis2012proof}. 
	
	Let $G\leq \sym(\Omega)$ be a finite transitive group. Since the elements of $G$ are permutations of $\Omega$, we say that $\mathcal{F} \subset G$ is intersecting if any two elements agree on an element of $\Omega$. We say that $G$ has the \emph{EKR property} if an intersecting set of $G$ has size at most the order of a point stabilizer.  Moreover, $G$ has the \emph{strict-EKR} property if it has the EKR property and any intersecting set of maximum size is a coset of a point stabilizer of $G$. The \emph{canonical intersecting sets} for the permutation group $G$ are the cosets of a point stabilizer, i.e., sets of the form $V_{\omega,\omega^\prime} = \{ g\in G \mid \omega^\prime = \omega^g \}$, where $\omega,\omega^\prime \in \Omega_q$.
	
	For both the $k$-subsets of $[n]$ and permutations of the symmetric group, the canonical intersecting sets are the only intersecting sets of maximum sizes.	We show that for some transitive permutation groups, if there are maximum intersecting sets other than the canonical ones, then a Hilton-Milner type result is impossible. This paper is concerned with a Hilton-Milner type result for the permutation group $\GL{q}$ acting naturally by left multiplication on $\Omega_q := \mathbb{F}_q^2 \setminus \{0\}$, where $q$ is a prime power. 
	
	Recall that the \emph{Singer cycles} of $\GL{q}$ are matrices of order $q^2-1$. If $A \in \GL{q}$ is a Singer cycle and $H = \langle A\rangle$, then $H$ is a regular subgroup of $\GL{q}$ acting on $\Omega_q$. Using the Clique-Coclique bound (see \cite[Theorem~2.1.1]{godsil2016erdos}), we can deduce that an intersecting set of $G$ has size at most $\frac{|\GL{q}|}{|H|} = q(q-1)$, which is the order of a point stabilizer of $\GL{q}$. In other words, $\GL{q}$ has the EKR property.
	
	Consider the affine plane $\ag(2,q)$, where $q$ is a prime power. Recall that $\ag(2,q)$ is the incidence structure whose points consist of $\mathbb{F}_q^2$ and whose lines consist of the sets $L_{u,v} = \left\{ tu+v \mid t\in \mathbb{F}_q \right\}$, with $v\in \mathbb{F}_q$ and $u\in \mathbb{F}_q^2 \setminus \{0\}$. It is easy to see that $\GL{q}$ acts (intransitively) on both the points and the lines of $\ag(2,q)$.  
	
	First, note that $\GL{q}$ has two orbits in its action on the lines of $\ag(2,q)$. The first orbit $\mathcal{O}_1$ is the $1$-dimensional subspaces of $\mathbb{F}_q^2$ (i.e., points of $\pg(1,q)$); so $|\mathcal{O}_1|= q+1$. The second orbit $\mathcal{O}_2$ is the set of all lines of $\ag(2,q)$ that do not contain the vector $0$. These lines are obtained by taking all the parallel lines (i.e., cosets) of each line in $\mathcal{O}_1$. By Lagrange's theorem  on the group $\mathbb{F}_q^2$, there are $q$ cosets for each line in $\mathcal{O}_1$. Hence, there are $q-1$ parallel lines to each element of $\mathcal{O}_1$. In other words, $|\mathcal{O}_2| = (q-1)(q+1)$.
	
	Let $\ell$ be a line of $\ag(2,q)$ which is a $1$-dimensional subspace spanned by $ [1,0]^T$. Let $\Delta = \left\{\ell,b_1+\ell,b_2+\ell,\ldots, b_{q-1} + \ell\right\}$ be the set of all parallel lines to $\ell$ and $K$ be the pointwise stabilizer of $\Delta$.
	Consider the subgroup of $\GL{q}$ given by 
	\begin{align*}
		M = \left\{ \begin{bmatrix}
			k & a\\
			0 & 1
		\end{bmatrix}
		: k\in \mathbb{F}_q^*,\ a\in \mathbb{F}_q \right\}.
	\end{align*}
	It is easy to see that $M \leq K$ because for any $A = \begin{bmatrix}
	k & a\\
	0 & 1
	\end{bmatrix}$, $\ell$ is an eigenspace corresponding to the eigenvalue $k$, thus $A\ell = \ell$, and $(A-I)b_i \in \ell$, for any $i\in \{1,2,3,\ldots,q-1\}$. In fact, one can prove that $M = K$.
	 Let $N$ be the setwise stabilizer of $\ell^\prime = b_1+\ell \in \mathcal{O}_2$ in $\GL{q}$. Since the action of $\GL{q}$ on the orbits of lines that do not contain the vector $0$ is transitive, by the orbit-stabilizer lemma, we deduce that 
	\begin{align*}
		|N| = \frac{|\GL{q}|}{|\mathcal{O}_2|} = \frac{(q^2-1)(q^2-q)}{q^2-1} = q(q-1).
	\end{align*}
	Since $K\leq N$, we conclude that $q(q-1) = |M| \leq |K| \leq |N| = q(q-1)$, and $M = N = K$.
	
	It is easy to see that $M$ is intersecting (in the action of $\GL{q}$ on the non-zero vectors) because if $k,k^\prime \in \mathbb{F}_q^*$ and $a,a^\prime \in \mathbb{F}_q$, then 
	\begin{align*}
		\begin{bmatrix}
			k & a\\
			0 & 1
		\end{bmatrix}^{-1}
		\begin{bmatrix}
			k^\prime & a^\prime \\
			0 & 1
		\end{bmatrix} = 
		\begin{bmatrix}
			k^{-1} & -k^{-1}a\\
			0 & 1
		\end{bmatrix}
		\begin{bmatrix}
			k^\prime & a^\prime \\
			0 & 1
		\end{bmatrix} =
		\begin{bmatrix}
			k^{-1}k^\prime & k^{-1}(a^{\prime}-a ) \\
			0 & 1
		\end{bmatrix}.
	\end{align*}
	As this matrix has an eigenvalue equal to $1$, it is not a derangement of $\GL{q}$ acting on $\Omega_q$. In other words, $M$ is an intersecting subgroup of maximum size in $\GL{q}$. 
	
	Since the stabilizers of lines from $\mathcal{O}_2$ are conjugate, it follows that the stabilizer of any line in $\mathcal{O}_2$ is intersecting of maximum size. Therefore, the stabilizer of a line of $\ag(2,q)$ that does not contain $0$ is another maximum intersecting set of $\GL{q}$. Hence, $\GL{q}$ acting on $\Omega_q$ does not have the strict-EKR property. 
	
	The following theorem was recently proved by Ahanjideh \cite{ahanjideh2021largest}. It was also partially obtained by Meagher and the second author in \cite{meagher2021erd}.
	
	\begin{thm}
		Let $q$ be a prime power. If $\mathcal{F} \subset \GL{q}$ is intersecting, then $|\mathcal{F}|\leq q(q-1)$. Moreover, equality holds if $\mathcal{F}$ is a coset of a point stabilizer or a coset of the stabilizer of a line in $\mathcal{O}_2$.\label{thm:EKR-GL}
	\end{thm}
	
	In \cite{ahanjideh2021largest}, it was claimed that if $\mathcal{F} \subset \GL{q}$ is an intersecting set which is not contained in a coset of a point stabilizer nor a coset of a stabilizer of a line in $\mathcal{O}_2$, then $|\mathcal{F}| \leq (q-1)(q-2)+1$ and this bound is sharp. Unfortunately, we were able to find counterexamples to this claim using \verb|Sagemath| \cite{sagemath}. Note that if such intersecting sets as claimed in \cite{ahanjideh2021largest} exist, then a maximal intersecting set with the aforementioned property will also exist. However, all maximal intersecting sets of $\GL{q}$, for $q \in \{3,4,5\}$, are all of maximum size; i.e., $q(q-1)$. Therefore, they are one of the two families of maximum intersecting sets described in Theorem~\ref{thm:EKR-GL}. The main result of this paper addresses this issue. 
	\begin{thm}
		Let $q$ be a prime power. If $\mathcal{F} \subset \GL{q}$ is intersecting, then $\mathcal{F}$ is contained in a coset of a point stabilizer of $\GL{q}$ or in a stabilizer of a line of $\mathcal{O}_2$.
		\label{thm:main}
	\end{thm} 
	
	We deduce the following corollary.
	\begin{cor}
		If $\mathcal{F} \subset \GL{q}$ is a maximal intersecting set, then $\mathcal{F}$ is a maximum intersecting set.\label{cor:maximal-maximum}
	\end{cor}
	
	A direct consequence of Corollary~\ref{cor:maximal-maximum} is that no Hilton-Milner type result holds for $\GL{q}$.
	One of our arguments in the proof of Theorem~\ref{thm:main} also enables us to prove the Complete EKR theorem for $\GL{q}$.
	
	\begin{thm}
		Let $q$ be a prime power. If $H\leq \GL{q}$ is transitive on $\Omega_q$, then $H$ has the EKR property. Moreover, an intersecting set of $H$ is contained in a coset of a stabilizer of an element of $\Omega_q$ in $H$ or in a coset of a stabilizer of a line of $\mathcal{O}_2$ in $H$. \label{thm:main2}
	\end{thm}

	Our proof for Theorem~\ref{thm:main} and Theorem~\ref{thm:main2} are given in the next section.
	
\section{Proof of the main results}

	In this section, we consider intersecting sets of $\GL{q}$ that are not contained in any coset of a point stabilizer and prove that $\mathcal{F}$ is in a stabilizer of a line of $\mathcal{O}_2$. Let $\mathcal{F} \subset \GL{q}$ be such intersecting sets. Without loss of generality, we may assume that the identity matrix $I$ belongs to $\mathcal{F}$ since we can always shift $\mathcal{F}$ with the inverse of an element $A\in \mathcal{F}$ and consider $A^{-1}\mathcal{F} = \{ A^{-1}B \mid B\in \mathcal{F} \} \ni I$.

	Next, we prove the following lemma about change of basis (the same argument was used in \cite{ahanjideh2021largest}).
	
	\begin{lem}
		Let $\mathcal{F} \subset \GL{q}$ and let $u\mbox{ and }v$ be two linearly independent vectors of $\mathbb{F}_q^2$. Let $\mathcal{F}_B$ be set of all matrices of $\mathcal{F}$ written in the basis $B = \{u,v\}$. Then, $\mathcal{F}$ is intersecting if and only if $\mathcal{F}_B$ is intersecting.
		\label{lem:change-of-basis}
	\end{lem}
	\begin{proof}
		Let $A,B \in \mathcal{F}$. If $P$ is the change-of-basis matrix corresponding to $\{u,v\}$, then, in the new basis, $A$ becomes $P^{-1}AP$ and $B$ becomes $P^{-1}BP$. Therefore, if there exists $u
		\in \Omega_q$ such that $A^{-1}Bu = u$ then $(P^{-1}A^{-1}P) P^{-1} BP = P^{-1} A^{-1}BP$ fixes $P^{-1}u$. Conversely, if $v\in \Omega_q$ is fixed by $P^{-1} A^{-1}BP$, then $Pv$ is fixed by $A^{-1}B$. Consequently, $\mathcal{F}$ is intersecting if and only if $\mathcal{F}_B$ is intersecting.
	\end{proof}
	
	Since $\mathcal{F}$ is non-canonical, the following holds
	\begin{align}
		\forall u\in \Omega_q,\ (\exists A_u \in \mathcal{F})  \mbox{ such that } (A_u u \neq u).\label{eq:main-formula}
	\end{align}
	Moreover, since $\mathcal{F}$ is intersecting and $I\in \mathcal{F}$, any non-identity element of $\mathcal{F}$ must fix an element of $\Omega_q$. That is, a non-identity element has an eigenvalue equal to $1$.
	\begin{prop}
		If $A\in \GL{q} \setminus \{I\}$ fixes a point, then it fixes exactly $q-1$ points.\label{prop:fixed-points}
	\end{prop}
	\begin{proof}
		If there exists $x\in \Omega_q$ such that $Ax =x$, then $x$ is an eigenvector with eigenvalue $1$ for $A$. Thus, $A$ is similar to one of
		\begin{align*}
			\begin{bmatrix}
				1 & 1\\
				0 &1
			\end{bmatrix} \mbox{ or }
			\begin{bmatrix}
				1 & 0 \\
				0 & \lambda
			\end{bmatrix},\mbox{ where }\lambda\neq 1.
		\end{align*}
		The former has only one eigenspace (containing $q-1$ eigenvectors), and the latter has two eigenspa\-ces with eigenvalues $1$ and $\lambda$.  The $q-1$ eigenvectors from the eigenspace with eigenvalue $1$ are fixed by $A$.
	\end{proof}

	Using \eqref{eq:main-formula}, for any $u\in \Omega_q$, we can find $A_u\in \mathcal{F}$ satisfying $A_u u \neq u$. If $u^\prime = ku$, for some $k\in \mathbb{F}_q^*$, then 
	\begin{align*}
		A_u u^\prime &= A_u (ku) = k A_u u\neq ku = u^\prime.
	\end{align*}
	Therefore, if $A_u$ does not fix $x$, then it does not fix any point of the line $\langle x\rangle$. Hence, we may consider only $q+1$ points from each of the $q+1$ lines in $\mathcal{O}_1.$ We let $u_1,u_2,\ldots,u_{q+1}$ be a set of pairwise non-colinear points of $\Omega_q$ (i.e., they are representatives of the lines of $\mathcal{O}_1$).
	
	\begin{defn}
		A \emph{base} of the intersecting set $\mathcal{F}$ is a set $\mathcal{B}(\mathcal{F})$ satisfying the following: 
		\begin{enumerate}[(i)]
			\item for every $i \in \{1,2,\ldots, q+1\}$, there exists $A_i \in \mathcal{B}(\mathcal{F})$ such that $A_i u_i \neq u_i$;\label{first}
			\item $\mathcal{B}(\mathcal{F})$ is minimal in the sense of \eqref{first}. That is, if $S \subsetneq C_q(\mathcal{F})$, then there exists $u \in \Omega_q$ such that for all $A\in S$, $A u = u$ (i.e., $S$ is canonical).
		\end{enumerate}\label{def:basis}
	\end{defn}
	
	Since $I\in \mathcal{F}$ and $I$ fixes every element of $\Omega_q$, $I$ does not belong to any base of an intersecting set. An intersecting set of $\GL{q}$ need not have a base. For instance, if $\mathcal{F}$ is the stabilizer of a point of $\Omega_q$, then $\mathcal{F}$ does not have a base. However, an intersecting set satisfying \eqref{eq:main-formula} must have a base.
	
	The following lemma is straightforward.
		
	\begin{lem}
		An intersecting set of $\GL{q}$ that does not contain a base is contained in a coset of a stabilizer of a point.\label{lem:base}
	\end{lem}
	
	\begin{obs}
		A consequence of Lemma~\ref{lem:base} is that if $S$ is a proper subset of a base of $\mathcal{F}$, then there exists $u\in \Omega_q$ such that for all $A\in S$, we have $A u =u$. In particular, $A$ fixes the line $\langle u\rangle$ pointwise.\label{observation}
	\end{obs}

	\begin{exa}
		Consider the subset of $\GL{5}$ given by
		\begin{align*}
			\mathcal{F} = \left\{ 
			\begin{bmatrix}
				1 & 0\\
				0 & 1
			\end{bmatrix},
			\begin{bmatrix}
				2 & 1\\
				0 & 1
			\end{bmatrix},
			\begin{bmatrix}
				1 & 1\\
				0 & 1
			\end{bmatrix}
			,
			\begin{bmatrix}
			3 & 1\\
			0 & 1
			\end{bmatrix}			
			 \right\}.
		\end{align*}
		Note that a necessary condition for $\mathcal{F}$ to be intersecting is that all matrices have eigenvalue equal to $1$. It is easy to see that $\mathcal{F}$ is a non-canonical intersecting set. Moreover, the set  
		\begin{align*}
			\mathcal{B}(\mathcal{F})  = \left\{ 
			\begin{bmatrix}
				2 & 1\\
				0 & 1
			\end{bmatrix},
			\begin{bmatrix}
			1 & 1\\
			0 & 1
			\end{bmatrix}			
			\right\}
		\end{align*}
		is a base of $\mathcal{F}$. Indeed, the matrix $A_1 = \begin{bmatrix}
		2 & 1\\
		0 & 1
		\end{bmatrix}$ has an eigenvalue equal to $1$ and the corresponding $1$-dimensional eigenspace $\ell_1 = \langle [1,-1]^T\rangle$ is fixed pointwise. Additionally, the matrix $A_2 = \begin{bmatrix}
		1 & 1\\
		0 & 1
		\end{bmatrix}$ fixes pointwise its eigenspace $\ell_2 = \langle [1,0]^T\rangle$ corresponding to the eigenvalue $1$. Therefore, if $x \not\in \ell_1\cup \ell_2$, then $x$ is not fixed by any matrix in $\mathcal{B}(\mathcal{F})$. If $x \in \ell_1$, then $A_2 x \neq x$ and if $x\in \ell_2$, then $A_1x \neq x$. 
		
		It is not hard to see that the set 
		\begin{align*}
		\mathcal{B}(\mathcal{F})  = \left\{ 
		\begin{bmatrix}
		2 & 1\\
		0 & 1
		\end{bmatrix},
		\begin{bmatrix}
		3 & 1\\
		0 & 1
		\end{bmatrix}			
		\right\}
		\end{align*}
		is another base of $\mathcal{F}$.
	\end{exa}

	This example shows that a non-canonical intersecting set can have more than one base and the size of a base of a non-canonical intersecting set can be as small as $2$. In the next lemma, we prove that a base of $\GL{q}$ has size $2$. 

	\begin{lem}
		Let $\mathcal{F}$ be a non-canonical intersecting set of $\GL{q}$. Any base of $\mathcal{F}$ has size two.\label{lem:base-of-size2}
	\end{lem}
	\begin{proof}
		We prove this assertion by a minimal counterexample. Assume that $\mathcal{F}$ has a base $\mathcal{B}(\mathcal{F}) $ $=\{A_1,A_2,A_3\}$. Since $I\in \mathcal{F}$, then $A_i$ fixes an element of $\Omega_q$, for any $i\in \{1,2,3\}$. In particular, the matrix $A_i$ fixes a line of $\mathcal{O}_1$ pointwise, for any $i\in \{1,2,3\}$.
		
		Suppose that $A_1$ is not diagonalizable. Then $1$ is an eigenvalue of $A_1$ with algebraic multiplicity $2$. Let $u \in \Omega_q$ be such that $A_1 u = u$ and $\ell = \langle u\rangle$. Note that the matrices in $\{A_1,A_2\}$ must fix a common point by minimality of a base (see Definition~\ref{def:basis}). This common fixed point must be in $\ell$. Therefore, $A_1$ and $A_2$ both fix $u$. Similarly, $A_1$ and $A_3$ must fix $u$. In other words, $\{A_1,A_2,A_3\}$ is contained in the stabilizer of the vector $u$ in $\GL{q}$. This is a contradiction.
		
		Suppose that $A_1$ is diagonalizable. Since $A_1 \neq I$, the eigenvalues of $A_1$ are $1$ and $k \neq 1$. Let $u,v \in \Omega_q$ such that $A_1 u = u$, $A_1 v = kv$, $\ell_1 = \langle u\rangle$, and $\ell_2 = \langle v\rangle$. By definition of a base, $\{A_1,A_2\}$ must fix a vector $x \in \Omega_q$. It is obvious that $x \in \ell_1$ or $x \in \ell_2$. Since $A_1 x = x$, and $A_1$ fixes $\ell_1$ pointwise, we have $x\in \ell_1$. Therefore, $A_1$ and $A_2$ both fix $u$. Similarly, $A_1$ and $A_3$ also fix $u$. In other words, $\{A_1,A_2,A_3\}$ is in the stabilizer of $u$ in $\GL{q}$. This is also a contradiction.
	\end{proof}
	
	\begin{thm}
		If $\mathcal{F}$ is an intersecting set satisfying \eqref{eq:main-formula} and $I\in \mathcal{F}$, then any base  of $\mathcal{F}$ fixes a line in $\mathcal{O}_1$.\label{thm:bases}
	\end{thm}
	\begin{proof}
		Since $\mathcal{F}$ satisfies \eqref{eq:main-formula}, it admits a base of the form $\mathcal{B}(\mathcal{F}) = \{ A_1,A_2 \}$ (see Lemma~\ref{lem:base-of-size2}). By Observation~\ref{observation}, for any $i\in \{1,2\}$ the matrix of $A_i$ fixes pointwise a line which we denote  $\ell_i \in \mathcal{O}_1$.

		If $\ell_1 = \ell_2,$ then $\mathcal{B}(\mathcal{F})$ fixes the line $\ell_1$.
		Assume that the lines $\ell_1,\ell_2$ are distinct and $\ell_1 = \langle u\rangle$ and $\ell_2 = \langle v\rangle$. Since $A_1$ and $A_2$ fix pointwise $\ell_1$ and $\ell_2$ respectively, we know that $A_1 u = u$ and $A_2 v= v$. Let us represent the matrices $A_1$ and $A_2$ in the basis $\{u,v\}$. There exist $a,b,c,d \in \mathbb{F}_q$ such that
		\begin{align*}
			\begin{cases}
				A_1 u = u\\
				A_1 v = au +bv
			\end{cases}
			\mbox{ and }\ \ \ \
			\begin{cases}
				A_2 u = c u +dv\\
				A_2 v = v
			\end{cases}.
		\end{align*}
		
		In the basis $\{u,v\}$, the matrices $A_1$ and $A_2$ become
		\begin{align*}
			\begin{bmatrix}
				1 & a\\
				0 & b
			\end{bmatrix}
			\mbox{ and }
			\begin{bmatrix}
				c & 0\\
				d & 1
			\end{bmatrix},
		\end{align*}
		respectively. Without loss of generality (see Lemma~\ref{lem:change-of-basis}), we identify $A_1$ and $A_2$ with these matrices. Since $\mathcal{B}(\mathcal{F})$ is intersecting, there exists $w = [w_1,w_2]^T \in \mathbb{F}_q^2 \setminus \{0\}$ such that $A_1w = A_2 w$. Therefore,
		\begin{align*}
			\begin{bmatrix}
			1 & a\\
			0 & b
			\end{bmatrix}
			\begin{bmatrix}
				w_1\\
				w_2
			\end{bmatrix}
			=
			\begin{bmatrix}
			c & 0\\
			d & 1
			\end{bmatrix}
			\begin{bmatrix}
				w_1\\
				w_2
			\end{bmatrix}
			.
		\end{align*}
		Equivalently, we have
		\begin{align*}
			\begin{bmatrix}
			c^{-1} & 0\\
			-dc^{-1} & 1 
			\end{bmatrix}
			\begin{bmatrix}
			1 & a\\
			0 & b
			\end{bmatrix}
			\begin{bmatrix}
			w_1\\
			w_2
			\end{bmatrix}
			&=
			\begin{bmatrix}
				c^{-1} & ac^{-1}\\
				-dc^{-1} & -dc^{-1}a +b
			\end{bmatrix}
			\begin{bmatrix}
				w_1\\
				w_2
			\end{bmatrix}
			=
			\begin{bmatrix}
			w_1\\
			w_2
			\end{bmatrix}
			.
		\end{align*}
		This implies that the matrix
		\begin{align*}
		\begin{bmatrix}
		1-c & a\\
		-d & cb-da-c
		\end{bmatrix},
		\end{align*}
		does not have full rank. In other words, its determinant is equal to $0$. Hence, we have $c( b-1 +ad +(1-b)c ) = 0$. From this, we get the following identity
		\begin{align}
			ad = (b-1)(c-1)\label{eq:first}.
		\end{align}
		
		\vspace*{0.3cm}
		\noindent{\bf Claim~1. If one of the matrices of $\mathcal{B}(\mathcal{F})$ is not diagonalizable, then $\mathcal{B}(\mathcal{F})$ fixes a line in $\mathcal{O}_1$.}
		
		\vspace*{0.3cm}
		
		If a matrix of $\GL{q}$ admitting an eigenvalue equal to $1$ is not diagonalizable, then the other eigenvalue must be equal to $1$, as well.
		Hence, $b=1$ or $c=1$ (see the expression of $A_1$ and $A_2$ in the basis $\{u,v\}$). By \eqref{eq:first}, we have $a=0$ or $d = 0$. Therefore, either $A_1v = bv$ or $A_2 u = cu$. In particular, if one of the matrices in $\mathcal{B}(\mathcal{F})$ belongs to $\SL{q}$, then $\mathcal{B}(\mathcal{F})$ has to fix a line.  
		Consequently, $\mathcal{B}(\mathcal{F})$ fixes a line in $\mathcal{O}_1$. 
		
		\vspace*{0.3cm}
		\noindent {\bf Claim~2. If both matrices of $\mathcal{B}(\mathcal{F})$ are diagonalizable, then it fixes a line in $\mathcal{O}_1$.}
		
		\vspace*{0.3cm}
		We assume henceforth that $A_1$ and $A_2$ are both diagonalizable. Then, there exist two pairs $(k,u^\prime),\ (t,v^\prime)\in \mathbb{F}_q\setminus \{0,1\} \times \Omega_q$ such that $A_1u^\prime = ku^\prime $ and $A_2 v^\prime = tv^\prime$. In other words, the eigenspaces of $A_1$ are $\ell_1$ and $\langle u^\prime\rangle$, and the ones for $A_2$ are $\ell_2$ and $\langle v^\prime\rangle$. Next, we prove that $\langle u^\prime \rangle = \langle v^\prime\rangle$.
		
		Let us write the vectors $u^\prime$ and $v^\prime$ in the basis $\{u,v\}$. There exist $\alpha,\beta,\gamma,\delta \in \mathbb{F}_q$ such that
		\begin{align*}
			\begin{cases}
				u^\prime &= \alpha u+ \beta v\\
				v^\prime &= \gamma u + \delta v.
			\end{cases}
		\end{align*}
		Note that $\beta \neq 0$ and $\gamma \neq 0$ since $\langle u \rangle \neq \langle u^\prime\rangle$ and $\langle v \rangle \neq \langle v^\prime\rangle$, respectively. As $k$ is an eigenvalue of $A_1$ with eigenvector $u^\prime = \alpha u+ \beta v$, we have
		\begin{align*}
			A_1u^\prime &= A(\alpha u +\beta v)
			= (\alpha+\beta a)u + \beta bv = k\alpha u + k\beta v.
		\end{align*}
		Similarly, for $A_2$ we have
		\begin{align*}
			A_2 v^\prime &= A_2(\gamma u + \delta v) = \gamma c u + (\gamma d + \delta) v = t\gamma u +t\delta v.
		\end{align*}
		Consequently, 
		\begin{align*}
			\begin{cases}
				\alpha +\beta a &= k\alpha\\
				\beta b &= k\beta
			\end{cases} \mbox{ and }\ \ \
			\begin{cases}
				\gamma c &= t \gamma\\
				\gamma d + \delta &= t \delta.
			\end{cases}
		\end{align*}
		As $\beta \neq 0$ and $\gamma \neq 0$, we conclude that $b = k$ and $c = t$. Furthermore, we have 
		\begin{align}
			a = \frac{(k-1)\alpha}{\beta}\ \ \  \mbox{ and }\ \ \ 
			d = \frac{(t-1)\delta}{\gamma}.
		\end{align}
		Applying these identities on $A_1$ and $A_2$, we have
		\begin{align}
			A_1 &= 
			\begin{bmatrix}
				1 & \frac{(k-1)\alpha}{\beta}\\
				0 & k
			\end{bmatrix} \ \ \ \mbox{ and }\ \ \ 
			A_2 = 
			\begin{bmatrix}
				t & 0\\
				\frac{(t-1)\delta}{\gamma} & 1
			\end{bmatrix}.\label{eq:final-form-matrices}
		\end{align} 
		Recall that the matrices $A_1$ and $A_2$ agree on $w = [w_1,\ w_2]^T$. Hence, $\left(A_2^{-1}A_1-I\right) w = 0$. That is, $A_2^{-1}A_1-I$ has determinant equal to $0$. We let the reader verify that
		\begin{align*}
			A_2^{-1}A_1-I &=
			\begin{bmatrix}
				1-t & \frac{(k-1)\alpha}{\beta}\\
				\frac{(1-t)\delta}{\gamma} & \frac{(k-1) \left( (1-t)\alpha \delta + \beta \gamma t  \right)}{\beta \gamma}  
			\end{bmatrix}.
		\end{align*}
		Since $\det ( A_2^{-1}A_1 - I) = 0$, we obtain
		\begin{align*}
			(1-t)(k-1) \left( (1-t)\alpha \delta + \beta \gamma t - \alpha \delta \right) = 0.
		\end{align*}
		Note that $t\neq 1$ and $k\neq 1$, since $A_1$ and $A_2$ are diagonalizable and not equal to the identity.  Hence, we can reduce the above equation to
		\begin{align}
			\beta \gamma = \alpha \delta.\label{eq:last-identity}
		\end{align}
		Using \eqref{eq:final-form-matrices}, it is not hard to verify that the eigenspaces of $A_1$ are the $1$-dimensional subspaces $\ell_1 = \langle u\rangle$ and $\langle u^\prime \rangle$, where
		\begin{align*}
			u^\prime = 
			\begin{bmatrix}
				\frac{\alpha}{\beta}\\
				1
			\end{bmatrix}.
		\end{align*}
		The eigenspaces of $A_2$ are $\ell_2 = \langle v\rangle$ and $\langle v^\prime\rangle$, where 
		\begin{align*}
				v^\prime &=
				\begin{bmatrix}
					1 \\
					\frac{\delta}{\gamma}
				\end{bmatrix}.
		\end{align*}
		Using \eqref{eq:last-identity}, it is easy to see that $u^\prime = \frac{\alpha}{\beta} v^\prime$, thus $\langle u^\prime\rangle = \langle v^\prime\rangle$.
		
		Consequently, the matrices $A_1$ and $A_2$ fix the line $\langle u^\prime\rangle = \langle v^\prime\rangle$. This completes the proof.
	\end{proof}

	\begin{lem}
		If $\mathcal{F} \subset \GL{q}$ fixes a line of $\mathcal{O}_1$, then it fixes a line in $\mathcal{O}_2$.\label{lem:fix-line-line}
	\end{lem}
	\begin{proof}
		Assume that $\ell \in \mathcal{O}_1$ is fixed by all matrices of $\mathcal{F}$. Since $\mathcal{F}$ lies in the stabilizer of the line $\ell$, we may conjugate (i.e., change basis) $\mathcal{F}$ so that its elements are in the stabilizer of the line generated by $[1,0]^T$. So, we may assume that $\ell$ is the $1$-dimensional subspace generated by $[1,0]^T$. Therefore,
		\begin{align*}
			\mathcal{F} \subset 
			\left\{
			\begin{bmatrix}
				x & y \\
				0 &1
			\end{bmatrix} : x\in x\in \mathbb{F}_q^*, y\in \mathbb{F}_q \right\}.
		\end{align*}
		We claim that the elements of $\mathcal{F}$ fix the line $\ell^\prime := [0,1]^T + \ell  .$ An element $A\in \mathcal{F}$ fixes $\ell^\prime$ if and only if 
		\begin{align*}
			(A-I)
			\begin{bmatrix}
				0\\
				1
			\end{bmatrix}
			\in \ell.
		\end{align*}
		For any $x\in \mathbb{F}_q^*,\ y\in \mathbb{F}_q$ such that the matrix
		\begin{align*}
			A = 
			\begin{bmatrix}
				x & y\\
				0 & 1
			\end{bmatrix} \in \mathcal{F},
		\end{align*}
		we have
		\begin{align*}
			(A-I)
			\begin{bmatrix}
			0\\
			1
			\end{bmatrix}
			=
			\begin{bmatrix}
				x-1 & y\\
				0 & 0
			\end{bmatrix}
			\begin{bmatrix}
			0\\
			1
			\end{bmatrix}
			= 
			\begin{bmatrix}
			y\\
			0
			\end{bmatrix} \in \ell.
		\end{align*}
		Consequently, the elements of $\mathcal{F}$ fix the line $\ell^\prime$.
	\end{proof}

	Now, we are ready to prove Theorem~\ref{thm:main}.
	
	\begin{proof}[Proof of Theorem~\ref{thm:main}]
		Let $\mathcal{F}$ be a non-canonical intersecting set of $\GL{q}$. We assume that $I\in \mathcal{F}$. We will prove that the elements of $\mathcal{F}$ must fix a common line of $\mathcal{O}_1$. Then by Lemma~\ref{lem:fix-line-line}, we deduce that they also must fix a line in $\mathcal{O}_2$.
		
		Since $\mathcal{F}$ is a non-canonical intersecting set of $\GL{q}$, it contains a base $\mathcal{B}(\mathcal{F})$. By Theorem~\ref{thm:bases}, $\mathcal{B}(\mathcal{F})$ fixes a line in $\mathcal{O}_1$. Assume that $\mathcal{F}$ contains a matrix $A$ that is not diagonalizable. If $u \in \Omega_q$ is such that $Au = u$, then $\ell = \langle u\rangle$ is the unique line in $\mathcal{O}_1$ fixed by $A$. Then for every $B\in \mathcal{F} \setminus \{A\}$, either $\{A,B\}$ is a base of $\mathcal{F}$ or it is in a stabilizer of an element of $\Omega_q$. In both cases, $A$ and $B$ have a common fixed line. Hence, $B\ell = \ell$. We conclude that every element of $\mathcal{F}$ fixes the line $\ell \in \mathcal{O}_1$.
		
		Next, we assume that all matrices in $\mathcal{F}$ are diagonalizable. Let $\mathcal{F} = \left\{ A_1,A_2,\ldots,A_t \right\}\cup \{I\}$, for some positive integer $t$. For any $i\in \{1,2,\ldots,t\}$, the matrix $A_i$ has two distinct eigenvalues, one of which is equal to $1$. Moreover, we let $\ell_1^{(i)} \in \mathcal{O}_1$ and $\ell_2^{(i)} \in \mathcal{O}_1$ be the two eigenspaces corresponding to $1$ and the other eigenvalue, for any $i\in \{1,2,\ldots,t\}$ (i.e., $\ell_1^{(i)}$ is fixed pointwise). Consider the graph $K(\mathcal{F})$ defined as follows:
		\begin{itemize}
			\item the vertex set of $K(\mathcal{F})$ is the set of all pairs $\left\{ \ell_1^{(i)},\ell_2^{(i)} \right\}$, for any $i\in \{1,2,\ldots,t\}$;
			\item two vertices $\left\{ \ell_1^{(i)},\ell_2^{(i)} \right\}$ and $\left\{ \ell_1^{(j)},\ell_2^{(j)} \right\}$ are adjacent if and only if $$\left\{ \ell_1^{(i)},\ell_2^{(i)} \right\} \cap \left\{ \ell_1^{(j)},\ell_2^{(j)} \right\} = \varnothing.$$
		\end{itemize}
		We note that $K(\mathcal{F})$ is an induced subgraph of the graph $K(\mathcal{O}_1,2)$ whose vertex set is the $2$-subsets of $\mathcal{O}_1$ and where a pair of $2$-subsets of $\mathcal{O}_1$ are adjacent if they do not share a common line. It is straightforward that the graph $K(\mathcal{O}_1,2)$ is isomorphic to the Kneser graph $K(q+1,2)$. Recall that the maximum cocliques of $K(q+1,2)$ are the canonical intersecting sets given by the EKR theorem. Moreover, the size of the largest cocliques that are non-canonical are described by the Hilton-Milner theorem \cite{hilton1967some}. In particular, the maximum cocliques of $K(q+1,2)$ have size $q$ and the non-canonical cocliques are of size at most $3$ (see \eqref{eq:HM}).

		Since $\mathcal{F}$ is a non-canonical intersecting set, any two $A_i$ and $A_j$, for distinct $i, j \in \{1,2,\ldots,t\}$, either form a base or are contained in a stabilizer of an element of $\Omega_q.$ In both cases, they have a common fixed line belonging to $\mathcal{O}_1$ (see Theorem~\ref{thm:bases}). Therefore, for any distinct $i,j \in \{1,2,\ldots,t\}$, we have $$\left\{ \ell_1^{(i)},\ell_2^{(i)} \right\} \cap \left\{ \ell_1^{(j)},\ell_2^{(j)} \right\} \neq \varnothing.$$
		Consequently, the graph $K(\mathcal{F})$ is a coclique of $K(q+1,2)$. By the Hilton-Milner theorem, we know that if the vertices of $\mathcal{F}$ are not contained in any canonical intersecting set of $K(q+1,2)$, then $|\mathcal{F}| \leq 3$ (see \eqref{eq:HM}). If $|\mathcal{F}|\geq 4$, then $\mathcal{F}$ must be in a canonical intersecting set. That is, there exists $i\in \{1,2,\ldots,t\}$ such that $\mathcal{F}$ fixes the line $\ell_1^{(i)}$ or $\ell_2^{(i)}$. If $|\mathcal{F}|\leq 3$, then there exist $A_1,A_2$ such that $\mathcal{F} = \{I,A_1,A_2\}$. By Theorem~\ref{thm:bases}, $\mathcal{F}$ clearly fixes a line in $\mathcal{O}_1$ (when $A_1$ and $A_2$ form a base) or is in a stabilizer of an element of $\Omega_q$. The latter however cannot happen since $\mathcal{F}$ would be a canonical intersecting set. 
		
		We conclude that in all cases, an intersecting set of $\GL{q}$ which is non-canonical must fix a line in $\mathcal{O}_2$.
		
	\end{proof}
	
	\begin{proof}[Proof of Theorem~\ref{thm:main2}]
		Let $H\leq \GL{q}$ be a transitive subgroup and let $\mathcal{F} \subset H$ be intersecting. Given $u \in \Omega_q$ and a line $\ell \in \mathcal{O}_1 \cup \mathcal{O}_2$, we let $H_u$ and $H_\ell$ be the stabilizers of $u$ and $\ell$ in $H$, respectively. Assume that $I\in \mathcal{F}$. Since $\mathcal{F} \subset H \leq \GL{q}$, we know by Theorem~\ref{thm:main} that $\mathcal{F}$ is in a stabilizer of a point of $\Omega_q$ in $\GL{q}$ or it is contained in a stabilizer of a line of $\mathcal{O}_2$ in $\GL{q}$. That is, there exist $u \in \Omega_q$ or $\ell \in \mathcal{O}_2$ such that $\mathcal{F} \subset H_u$ or $\mathcal{F} \subset H_\ell$. 
		Consequently, if $H\leq \GL{q}$ is transitive and if $\mathcal{F} \subset H$ is intersecting, then 
		\begin{align*}
			|\mathcal{F}| \leq \max \left\{ |H_u|, |H_\ell| \right\}.
		\end{align*}
		It is obvious that $|H_u| = \frac{|H|}{|\Omega_q|}$. On the other hand, since $H$ is transitive on $\Omega_q$, $H$ is transitive on $\mathcal{O}_1$. Let $ a + \langle u \rangle$ and $ b + \langle v\rangle$ be two lines in $\mathcal{O}_2$. Then, $\{a,u\}$ and $\{b,v\}$ are both linearly independent. By transitivity of $H$ on $\Omega_q$, there exists $A\in H$ such that $Au = v$. Since $\{b,v\}$ is a basis of $\mathbb{F}_q^2$, there exist $\alpha,\beta \in \mathbb{F}_q$ such that $Aa = \alpha b + \beta v$. Note that $\alpha \neq 0$ since $a$ and $u$ are linearly independent (so their images by an invertible matrix are also linearly independent). Let $C = \alpha^{-1} A$. We have
		\begin{align*}
			\begin{cases}
				Cu &= \alpha^{-1}Au = \alpha^{-1}v,\\
				Ca &= \alpha^{-1} Aa = b + \alpha^{-1}\beta v.
			\end{cases}
		\end{align*}
		Consequently, $C\langle u\rangle = \langle v\rangle$ and $Ca -b \in \langle v \rangle $. In other words, $C(a+\langle u\rangle) = b+ \langle v\rangle$. Therefore, $H$ is transitive on the lines in $\mathcal{O}_2.$ In particular, $|H_\ell| = \frac{|H|}{|\mathcal{O}_2|} = \frac{|H|}{q^2-1}$.
		
		We conclude that if $\mathcal{F} \subset H$ is intersecting, then 
		\begin{align*}
		|\mathcal{F}| \leq \frac{|H|}{q^2-1}.
		\end{align*}
		In other words, $H$ has the EKR property.
	\end{proof}

\section{Future work}
In this paper, we proved that there is no Hilton-Milner theorem for the group $\GL{q}$ acting on $\Omega_q$, where $q$ is a prime power. In particular, we gave another proof that $\GL{q}$ acting on $\Omega_q$ has the EKR property and its maximum intersecting sets  are one of two families: the cosets of a point stablizer or the cosets of a stabilizer of a line in $\mathcal{O}_2$. A byproduct of our result in Theorem~\ref{thm:main} is that any transitive subgroup of $\GL{q}$ has the EKR property.

For any $n\geq 3$ and a prime power $q$, the group $\gl_n(\mathbb{F}_q)$ has the EKR property due to the existence of the Singer subgroups, which are regular subgroups. Since the coset of a stabilizer of a non-zero vector of $\mathbb{F}_q^n$ and the cosets of a stabilizer of a hyperplane of $\mathbb{F}_q^n$ are intersecting sets of maximum size, this action of $\gl_n(\mathbb{F}_q)$ does not have the strict-EKR property. It is therefore natural to ask whether our main result (Theorem~\ref{thm:main}) can be generalized to $\gl_n(\mathbb{F}_q)$ for $n\geq 3$. Using \verb|Sagemath|, we were able to verify that the answer to this question is negative for $\gl_3(\mathbb{F}_2)$.  

\begin{qst}
	Is there a Hilton-Milner type result for $\gl_n(\mathbb{F}_q)$ acting on the non-zero vectors of $\mathbb{F}_q^n$, for $n\geq 3$? 
\end{qst}

\noindent{\sc Acknowledgement.} We would like to thank Karen Meagher for proofreading and helping us improve the presentation of this paper.

\bibliographystyle{plain}

\end{document}